\documentclass{elsarticle}
\makeatletter
\def\ps@pprintTitle{%
 \let\@oddhead\@empty
 \let\@evenhead\@empty
 \def\@oddfoot{\centerline{\thepage}}%
 \let\@evenfoot\@oddfoot}
\makeatother

\usepackage[all,cmtip]{xy}
\usepackage{amsmath, amssymb}
\usepackage{graphics}
\usepackage{caption}
\usepackage{subcaption}
\usepackage{amsthm}
\usepackage{algorithm}
\usepackage{algorithmic}
\usepackage{booktabs}
\usepackage{color}

\def\e{{\rm e}}
\def\rmd{{\rm d}}

\newtheorem{theorem}{{\bf Theorem}}

\newtheorem{proposition}[theorem]{{\bf Proposition}}
\newtheorem{lemma}[theorem]{{\bf Lemma}}

\def\bfa{\boldsymbol{a}}
\def\bfb{\boldsymbol{b}}
\def\bfx{\boldsymbol{x}}
\def\bfy{\boldsymbol{y}}
\def\bfp{\boldsymbol{p}}

\def\CC{\mathbb{C}}
\def\RR{\mathbb{R}}

\def\CCC{\mathcal{C}}
\def\OOO{\mathcal{O}}
\def\tOOO{\tilde{\mathcal{O}}}

\def\SimilarPol{\texttt{SimilarPol}}
\def\SimilarGen{\texttt{SimilarGen}}

\def\Sage{\textsf{Sage}}
\def\Singular{\textsf{Singular}}
\def\Flint{\textsf{FLINT}}

\begin{document}

\begin{frontmatter}
\title{Detecting Similarity of Rational Plane Curves}


\author[a]{Juan Gerardo Alc\'azar\fnref{proy}}
\ead{juange.alcazar@uah.es}
\author[a]{Carlos Hermoso}
\ead{carlos.hermoso@uah.es}
\author[b]{Georg Muntingh\fnref{grant}}
\ead{georgmu@math.uio.no}

\address[a]{Departamento de F\'{\i}sica y Matem\'aticas, Universidad de Alcal\'a,
E-28871 Madrid, Spain}
\address[b]{SINTEF ICT, PO Box 124 Blindern, 0314 Oslo, Norway, and\\Department of Mathematics, University of Oslo, PO Box 1053, Blindern, 0316 Oslo, Norway\!\!\!\!\!\!\!\!
}

\fntext[proy]{Supported by the Spanish ``Ministerio de
Ciencia e Innovacion" under the Project MTM2011-25816-C02-01. Partially supported by
a Jos\'e Castillejos' grant from the Spanish Ministerio de Educaci\'on, Cultura y Deporte.
Member of the Research Group {\sc asynacs} (Ref. {\sc ccee2011/r34}) }

\fntext[grant]{Partially supported by the Giner de los R\'ios grant from the Universidad de Alcal\'a.}


\begin{abstract}
A novel and deterministic algorithm is presented to detect whether two given rational plane curves are related by means of a similarity, which is a central question in Pattern Recognition. As a by-product it finds all such similarities, and the particular case of equal curves yields all symmetries. A complete theoretical description of the method is provided, and the method has been implemented and tested in the {\Sage} system for curves of moderate degrees.
\end{abstract}

\end{frontmatter}

\section{Introduction}\label{section-introduction}
\noindent A central problem in Pattern Recognition and Computer Vision is to detect whether a certain object corresponds to one of the objects stored in a database. The goal is to identify the given object as one of the objects in the database, and therefore to classify it or to consider it as unknown. When the objects involved in the recognition process are planar and defined by their silhouettes, algebraic curves can be used. For instance, in this setting silhouettes of aircraft prototypes \cite{Huang} and of sea animals \cite{TC00} have previously been considered.

However, the two objects to be compared need not be in the same position, orientation, or scale. In order to compare the two objects, it should therefore be checked whether there exists a nontrivial movement, also called \emph{similarity}, transforming the object to be analyzed into the possible target in the database. This problem, known in the Computer Vision literature as \emph{pose estimation}, has been extensively considered using many different techniques, using B-splines \cite{Huang}, Fourier descriptors \cite{Sener}, complex representations \cite{TC00}, statistics \cite{Gal, Lei, Mishra} also in the 3D case, moments \cite{Suk1, Taubin2}, geometric invariants \cite{Unel, WU98a, WU98b}, Newton-Puiseux parametrizations \cite{MM02}, and differential invariants \cite{Boutin, Calabi, Weiss}. The interested reader may consult the bibliographies in these papers to find other references on the matter.

With exception for the references concerning the B-splines and differential invariants, the above methods use the implicit form of the curves. Moreover, the above methods are either numerical, or only efficient when considered in a numerical setting. The reason for this is that in Pattern Recognition it is often assumed that the inputs are ``fuzzy''. For instance, in many cases the objects are represented discretely as point clouds. In this situation, an implicit equation is usually first computed for each cloud, and the comparison is performed later. In other cases there may be occluded parts or noise in the input. Alternatively the input might be exact, but modeling a real object only up to a certain extent. In all of these cases we do not need an perfect matching, so that a numerical comparison is sufficient.

In this paper, we address the problem from a perspective that differs in two ways. First of all, we assume that the curves are given in \emph{exact} arithmetic, so that we can provide a deterministic answer to the question whether these two curves are similar. If so, our algorithm will find the similarities transforming one into the other. A second difference is that we start from rationally, or piecewise rationally, \emph{parametrized} curves and carry out all computations in the parameter space. As a consequence, the cost of converting to the implicit form, both in terms of computing time and growth of the coefficients, is avoided. This representation is important in applications of Computer-Aided Geometric Design. In CAD-based systems, for instance, curves are typically represented by means of piecewise-rational parametrizations, usually of moderate degrees, like (rational) B\'ezier curves, splines and NURBS. Therefore, if the comparison is to be made between these representations, an algorithm based on the parametric form is desirable.

A first potential application of the ideas in the paper is related to computer algebra systems. Assume that a database with classical curves is stored in your favourite computer algebra system. Using the algorithms in this paper, the system can recognize a certain curve introduced by the user as one of the curves in the database. This way, the user can identify a curve as, say, a cardioid, an epitrochoid, a deltoid, etc.

A second application is related to Computer Graphics. In this field, the recognition of similarities simplifies manipulating and storing images, since in the presence of a similarity we need to store only one image, and the similarity producing the other. In the case of curves represented by B\'ezier curves or splines, one can detect similarity by checking whether the corresponding control polygons are similar. This is easy, because the control polygons are piecewise linear objects. However, this is less clear in the case of rational B\'ezier curves or NURBS. Furthermore, if one is not interested in global similarities but in {\it partial} similarities, i.e., in determining whether parts of the objects are similar, then it is no longer clear how to derive this from the control polygons. In addition, any algorithm to find the similarities between rational curves is also an algorithm to find the symmetries of a rational curve, since an algebraic curve is self-similar if and only if it is symmetric; see Proposition \ref{prop:sym}. Symmetry detection has been massively addressed in the field of Computer Graphics to gain understanding when analyzing pictures, and also in order to perform tasks like compression, shape editing, and shape completion; see for instance \cite{Berner08, Bokeloh, Li08, Li10, Mitra06, Podolak}, and the references therein.

We exploit and generalize some ideas used in \cite{Alcazar13, Alcazar.Hermoso13} for the computation of symmetries of rational plane curves, improving the algorithm provided in \cite{Alcazar.Hermoso13}. Our approach exploits the rational parametrizations to reduce to calculations in the parameter domain, and therefore to operations on univariate polynomials. Thus we proceed symbolically to determine the existence and computation of such similarities, using basic polynomial multiplication, GCD-computations. Additionally, if a numerical approximation is desired, univariate polynomial real-solving must be used as well. We have implemented and tested this algorithm in the {\Sage} system \cite{sage}. It is also worth mentioning that, as a by-product, we achieve an algorithm to detect whether a given rational curve is symmetric and to find such symmetries. This problem has previously been studied from a deterministic point of view \cite{Alcazar13, Alcazar.Hermoso13, Carmichael, LR08}, and by many authors from an approximate point of view.

The structure of the paper is the following. Some generalities on similarities and symmetries, to be used throughout the paper, are established in Section \ref{sec-prelim}. The method itself is addressed in Section \ref{sec-detect-2}, first for polynomially parametrized curves and then for the general case. The case of piecewise rational curves is addressed at the end of this section. Finally, practical details on the implementation, including timings, are provided in Section \ref{sec-implem}.

\section{Similarities and symmetries} \label{sec-prelim}
\noindent Throughout the paper, we consider rational plane algebraic curves $\CCC_1, \CCC_2 \subset \RR^2$ that are neither a line nor a circle. Such curves are \emph{irreducible}, i.e., they are the zero sets of polynomials that can not be factored over the reals, and can be parametrized by rational maps
\begin{equation}\label{eq:parametrizations}
\phi_j: \RR \dashrightarrow C_j\subset \RR^2, \qquad \phi_j(t) = \big( x_j(t), y_j(t) \big),\qquad j = 1,2.
\end{equation}
The components $x_j, y_j$ of $\phi_j$ are rational functions of $t$, and they are defined for all but a finite number of values of $t$. We assume that the parametrizations \eqref{eq:parametrizations} are \emph{proper}, i.e., birational or, equivalently, injective except for perhaps finitely many values of $t$. In particular, the parametrizations have a rational inverse defined on their images. This is no restriction on the curves $\CCC_1,\CCC_2$, as any rational curve always admits a proper parametrization. For proofs of these claims and for a thorough study on properness, see \cite[\S 4.2]{SWPD}.

Roughly speaking, an \emph{(affine) similarity} of the plane is a linear affine map from the plane to itself that preserves ratios of distances. More precisely, a map $f: \RR^2 \longrightarrow \RR^2$ is a similarity if $f(x) = Ax + b$ for an invertible matrix $A\in \RR^{2\times 2}$ and a vector $b\in \RR^2$, and there exists an $r > 0$ such that
\[ \|f(x) - f(y)\|_2 = r\|x - y\|_2,\qquad x,y\in \RR^2, \]
where $\|\cdot\|_2$ denotes the Euclidean norm. We refer to $r$ as the \emph{ratio} of the similarity. Notice that if $r = 1$ then $f$ is an \emph{(affine) isometry}, in the sense that it preserves distances. The similarities of the plane form a group under composition, and the isometries form a subgroup. It is well known, and easy to derive, that a similarity can be decomposed into a translation, an orthogonal transformation, and a uniform scaling by its ratio $r$. The curves $\CCC_1, \CCC_2$ are \emph{similar}, if one is the image of the other under a similarity.

For analyzing the similarities of the plane, we have found it useful to identify the Euclidean plane with the complex plane. Through this correspondence $(x,y) \simeq x + iy$, the parametrizations \eqref{eq:parametrizations} correspond to parametrizations
\begin{equation}\label{eq:parametrizations2}
z_j: \RR\dashrightarrow C_j\subset \CC, \qquad z_j(t) = x_j(t) + iy_j(t),\qquad j = 1, 2.
\end{equation}
We can distinguish two cases for a similarity of the complex plane. A similarity $f$ is either \emph{orientation preserving}, in which case it takes the form $f(z) = \bfa z + \bfb$, or \emph{orientation reversing}, in which case it takes the form $f(z) = \bfa \overline{z} + \bfb$. In each case, its ratio $r = |\bfa|$.

A \emph{M\"obius transformation} is a rational function
\begin{equation}\label{eq:Moebius}
\varphi: \RR\dashrightarrow \RR,\qquad \varphi(t) = \frac{\alpha t + \beta}{\gamma t + \delta},\qquad \Delta := \alpha\delta - \beta\gamma \neq 0.
\end{equation}
It is well known that the birational functions on the line are the M\"obius transformations \cite{SWPD}.

\begin{theorem}\label{th-fund}
Let $\CCC_1, \CCC_2 \subset \CC$ be rational plane curves with proper parametrizations $z_1, z_2: \RR\dashrightarrow \CC$. Then $\CCC_1$, $\CCC_2$ are similar if and only if there exists a similarity $f$ and a M\"obius transformation $\varphi$ such that
\begin{equation}\label{ec-0}
z_2\big(\varphi(t)\big)= f\big(z_1(t)\big).
\end{equation}
Moreover, if $\CCC_1, \CCC_2$ are similar by a similarity $f$, then there exists a unique M\"obius transformation $\varphi$ satisfying \eqref{ec-0}.
\end{theorem}

Equation \eqref{ec-0} is equivalent to the existence of a commutative diagram
\begin{equation}\label{eq:fundamentaldiagram}
\xymatrix{
\CC \ar[r]^{f} & \CC \\
\RR \ar@{-->}[u]^{z_1} \ar@{-->}[r]_{\varphi} & \RR \ar@{-->}[u]_{z_2}
}
\end{equation}
which relates the problem of finding a similarity ``upstairs'' to finding a corresponding M\"obius transformation ``downstairs''.

\begin{proof}
Suppose that $\CCC_1$ and $\CCC_2$ are similar. Then there exists a similarity $f$ of the plane that restricts to a bijection $f: \CCC_1\longrightarrow \CCC_2$. Since $z_1, z_2$ are proper, their inverses $z_1^{-1}, z_2^{-1}$ exist as rational functions on $\CCC_1, \CCC_2$. The composition $z_2^{-1}\circ f\circ z_1$ is therefore also a rational function with inverse $z_1^{-1}\circ f^{-1}\circ z_2$, which is rational as well. Hence, $\varphi=z_2^{-1}\circ f \circ z_1$ is a birational transformation on the real line, implying that it must be a M\"obius transformation. Clearly this choice of $f$ and $\varphi$ makes the diagram \eqref{eq:fundamentaldiagram} commute.

Conversely, suppose we are given a similarity $f$ of the plane and a M\"obius transformation $\varphi$ that makes the diagram \eqref{eq:fundamentaldiagram} commute. The similarity $f$ maps any point $z\in \CCC_1$ to $f(z) = z_2\circ \varphi\circ z_1^{-1}(z)$, which lies in the image of $z_2$ and therefore on the curve $\CCC_2$. It follows that $\CCC_1$ and $\CCC_2$ must be similar.

For the final claim, suppose that there are two M\"obius transformations $\varphi_1, \varphi_2$ making the diagram \eqref{eq:fundamentaldiagram} commute. Then $z_2\big(\varphi_1(t)\big) = f\big(z_1(t)\big) = z_2\big(\varphi_2(t)\big)$. Since the parametrization $z_2$ is proper, we conclude $\varphi_1=\varphi_2$.
\end{proof}

Although the final claim refers to a M\"obius transformation $\varphi$ that is unique, its coefficients, i.e., $\alpha,\beta,\gamma,\delta$ in \eqref{eq:Moebius}, are only defined up to a common multiple. Since $\varphi$ maps the real line to itself, these coefficients can always be assumed to be real by dividing by a common complex number if necessary.

Depending on whether $f$ preserves or reverses the orientation of the complex plane, the diagram \eqref{eq:fundamentaldiagram} will either take the form
\begin{equation}\label{ec-1}
z_2\big(\varphi(t)\big)=\bfa z_1(t)+\bfb,
\end{equation}
or
\begin{equation} \label{ec-2}
z_2\big(\varphi(t)\big)=\bfa \overline{z_1(t)}+\bfb.
\end{equation}

Next, let $\CCC$ be a rational plane curve, neither a line nor a circle, given by a proper rational map
\[ \phi: \RR\dashrightarrow \CCC\subset \RR^2,\qquad \phi(t) = \big(x(t), y(t)\big). \]
A \emph{self-similarity} $f$ of $\CCC$ is a similarity of the plane satisfying $f(\CCC) = \CCC$.

\begin{proposition}\label{prop:sym}
Let $\CCC$ be an algebraic curve that is not a union of (possibly complex) concurrent lines. Any self-similarity $f$ of $\CCC$ is an isometry.
\end{proposition}
\begin{proof}
Writing $z = x + iy$, the curve $\CCC$ is the zeroset of a polynomial
\begin{equation*}
G(z,\bar{z}) = \sum_{k=0}^d G_k(z, \bar{z}),
\end{equation*}
where $G_k(z, \bar{z})$ is homogeneous of degree $k$ and $G_d(z, \bar{z})$ is nonzero. Moreover, since $\CCC$ is not the union of concurrent lines, at least one other form $G_l(z, \bar{z})$, with $0\leq l\leq d-1$, must be nonzero.

Let us first consider the case that $f(z) = \bfa z + \bfb$ is an orientation-preserving symmetry. Suppose that $f$ is not an isometry, i.e., $|\bfa|\neq 1$. Then $f$ has a unique fixed point $\bfb/(1 - \bfa)$, which we can without loss of generality assume to be located at the origin by applying an affine change of coordinates if necessary. Hence, in these coordinates $\CCC$ has the symmetry $f(z) = \bfa z$, so that it is also the zeroset of $G(\bfa z, \bar{\bfa} \bar{z})$, which therefore must be a scalar multiple of $G(z,\bar{z})$. Since this scalar multiple must be the same for $G_l$ and $G_d$, one has $\bfa^i \bar{\bfa}^{l-i} = \bfa^j \bar{\bfa}^{d-j}$ for some $i$ and $j$, implying $|\bfa|^l = |\bfa|^d$ and therefore contradicting our assumption that $|\bfa| \neq 1$.

Next consider the case that $f(z) = \bfa \bar{z} + \bfb$ is an orientation-reversing symmetry. Then it also has the (possibly trivial) orientation-preserving symmetry $f\big(f(z)\big) = \bfa\bar{\bfa} z + \bfb + \bfa\bar{\bfb}$, which must be an isometry by the above paragraph. It follows that $|\bfa|^2 = 1$ and therefore that $f$ must be an isometry.
\end{proof}

The isometries of the plane form a group under composition, which consists of \emph{reflections} $f(z) = \e^{i\theta}\big(\overline{z - \bfb}\big) + \bfb$, which reflect in the line $\Im\big(\e^{-i\theta/2}z - \e^{-i\theta/2} \bfb \big) = 0$, \emph{rotations} $f(z) = \e^{i\theta} (z - \bfb) + \bfb$, which rotate around a point $\bfb$, \emph{translations} $f(z) = z + \bfb$, and \emph{glide reflections}, which are a composition of a reflection and a translation \cite{Coxeter}.

An isometry of the plane leaving $\CCC$ invariant, is more commonly known as a \emph{symmetry} of $\CCC$.  When $\CCC$ is different from a line it cannot be invariant under a translation or a glide reflection, as this would imply the existence of a line intersecting the curve in infinitely many points, contradicting B\'ezout's theorem. The remaining symmetries are therefore the \emph{mirror symmetries} (reflections) and the \emph{rotation symmetries}. The special case of \emph{central symmetries} is of particular interest and corresponds to rotation by $\theta = \pi$. Notice that the identity map is a symmetry of any curve $\CCC$, called the \emph{trivial symmetry}.

Symmetries of algebraic curves and the similarities between them are related as follows.

\begin{theorem} \label{thsim}
Let $\CCC_1,\CCC_2$ be two distinct, irreducible, similar, algebraic curves, neither of them a line or a circle. Then there are at most finitely many similarities mapping $\CCC_1$ to $\CCC_2$. Furthermore, there is a unique similarity mapping $\CCC_1$ to $\CCC_2$ if and only if $\CCC_1$ (and therefore also $\CCC_2$) only has the trivial symmetry.
\end{theorem}
\begin{proof}
If $f, \tilde{f}$ are two similarities of the plane mapping $\CCC_1$ to $\CCC_2$, then $g := \tilde{f}^{-1}\circ f$ is also a similarity, which maps $\CCC_1$ into itself. As $\CCC_1$ is irreducible and not a line, Proposition~\ref{prop:sym} implies that $g$ is a symmetry of $\CCC_1$. Since an algebraic curve that is neither a line nor a circle has at most finitely many symmetries \cite[\S 5]{LRTh}, there can therefore be at most finitely many similarities mapping $\CCC_1$ to $\CCC_2$.

For the second claim, if $\CCC_1$ only has the trivial symmetry, $\tilde{f}^{-1}\circ f$ is the identity and $\tilde{f} = f$ is the unique similarity mapping $\CCC_1$ to $\CCC_2$. Conversely, if $\CCC_1$ has a nontrivial symmetry $\tilde{g}$, then $f \circ \tilde{g}$ is an additional similarity mapping $\CCC_1$ to $\CCC_2$.
\end{proof}

\section{Detecting similarity between rational curves} \label{sec-detect-2}
\noindent In this section we derive a procedure for detecting whether the curves $\CCC_1, \CCC_2$ are similar by an orientation preserving similarity. In this case, by Theorem \ref{th-fund}, Equation \eqref{ec-1} holds for some similarity $f(z)=\bfa z + \bfb$ and a M\"obius transformation $\varphi$. The case of an orientation reversing similarity is analogous, by replacing $z_2$ by $\overline{z_2}$.

We first attack the simpler case when $z_1, z_2$ are polynomial. After that we consider the case when either $z_1$ or $z_2$ is not polynomial, while distinguishing between $\delta\neq 0$ and $\delta=0$, with $\delta$ as in \eqref{eq:Moebius}. In each case our strategy is to eliminate $\bfa, \bfb$ and reduce to a simpler problem in the parameter space, i.e., downstairs in the diagram \eqref{eq:fundamentaldiagram}. Once we obtain the possible M\"obius transformations downstairs, they can be lifted to corresponding similarities upstairs.

\subsection{The polynomial case} \label{subsec-poly-case}
\noindent A curve is \emph{polynomial} if it admits a polynomial parametrization. It is easy to check whether a rational parametrization defines a polynomial curve, and to quickly compute a polynomial parametrization in that case; see \cite{SWPD} for both statements. As a consequence, we can assume polynomial curves to be polynomially parametrized, without loss of generality and without loss of significant computation power. Notice that if $\CCC_1, \CCC_2$ are similar and one of them is polynomial, then the other must be polynomial as well.

For polynomial parametrizations $z_1, z_2$, the right hand side in \eqref{ec-0} is polynomial, implying that the left hand side is polynomial as well. It follows that in that case the corresponding M\"obius transformation $\varphi$ should be linear affine, i.e., $\varphi(t) = \alpha t + \beta$. In this section we assume that the curves $\CCC_1,\CCC_2$ are related by a similarity whose corresponding M\"obius transformation is of this form. This includes similar polynomial curves as a special case, but also other rational curves needed for the case treated in Section \ref{subcase-2}.

We require some mild assumptions on our parametrization $z_1$, namely that $z_1$, and therefore also its derivatives $z_1', z_1''$, are well defined at $t = 0$ and that $z_1', z_1''$ are nonzero at $t = 0$. Notice that there are only finitely many parameters $t$ for which one of these conditions does not hold, so these conditions hold after applying an appropriate, or even random, linear affine change of the parameter $t$.

Evaluating \eqref{ec-1} at $t = 0$ yields
\begin{equation} \label{p-1}
z_2(\beta)=\bfa z_1(0)+\bfb.
\end{equation}
To get rid of $\bfb$, we differentiate \eqref{ec-1} and evaluate at $t = 0$, which gives
\begin{equation} \label{p-2}
z'_2(\beta)\alpha=\bfa z'_1(0).
\end{equation}
Differentiating \eqref{ec-1} twice and evaluating at $t=0$ yields
\begin{equation} \label{p-3}
z''_2(\beta)\alpha^2=\bfa z''_1(0).
\end{equation}
The unknowns $\bfa$ and $\alpha$ are nonzero since $f$ and $\varphi$ are invertible, and $z_2'(\beta), z_2''(\beta)$ are not identically zero because $\CCC_2$ is not a line. Dividing \eqref{p-3} by \eqref{p-2} yields
\begin{equation} \label{p-4}
\alpha=\displaystyle{\frac{z''_1(0)}{z'_1(0)}\cdot \frac{z'_2(\beta)}{z''_2(\beta)}},
\end{equation}
which does not involve $\bfa$ or $\bfb$.

Since the M\"obius transformation $\varphi(t) = \alpha t + \beta$ maps the real line to itself, both $\alpha$ and $\beta$ are required to be real. That is, their imaginary parts $\Im(\alpha), \Im(\beta)$ are zero. The following lemma shows that this limits $\beta$ to only finitely many candidates.

\begin{lemma} \label{imcero}
If $\Im(\alpha)=0$ holds for infinitely many values of $\beta$, then $\CCC_2$ is a line.
\end{lemma}

\begin{proof}
Writing $z_1''(0)/z_1'(0) = a + ib$ and $z_2(\beta) = x(\beta) + iy(\beta)$, the condition $\Im(\alpha)=0$ is equivalent to
\[ b\big(x'x''+y'y''\big) + a \big(x''y'-x'y''\big) = 0.\]
If $y'$ is identically $0$, then $\CCC_2$ is a line and the result follows. So assume that $y'$ is not identically $0$. Then the above condition is equivalent to
\[ b \frac{x'x''+y'y''}{y'^2} + a\left(\frac{x'}{y'}\right)' = 0.\]
Changing to a new variable $z=x'/y'$ and dividing by $z^2 + 1$, we obtain
\[ b \frac{zz'}{z^2 + 1} + a\frac{z'}{z^2 + 1} = - b\frac{y''}{y'}. \]
Integrating this equation yields
\begin{equation}\label{eq:lemimalpha}
\frac{b}{2}\mbox{ln}\left(x'^2+y'^2\right) = -a\arctan\left(\frac{x'}{y'}\right) + k,
\end{equation}
for some constant $k$. Since $z''_1(0)$ is nonzero, $a,b$ cannot both be zero. If $b=0$, then $x'/y'$ is constant and $\CCC_2$ is a line. If $b\neq 0$, writing $x' + i y' = r\e^{i\theta}$, Equation~\eqref{eq:lemimalpha} becomes $r = K\e^{-\theta a/b}$ for some nonzero constant $K$. If $a\neq 0$, this curve is a logarithmic spiral, which is a non-algebraic curve and therefore contradicts that $x'$ and $y'$ are rational. If $a = 0$, on the other hand, then the arc length $r = \sqrt{x'^2 + y'^2}$ of $z_2$ is constant, which again implies that $\CCC_2$ is a line \cite{FaroukiSakkalis91}.
\end{proof}

As a consequence, we obtain a polynomial condition $\xi(\beta)=0$ on $\beta$, where
\begin{equation}\label{eq:xi}
\xi(\beta) := b\big(x'(\beta)x''(\beta)+y'(\beta)y''(\beta)\big) + a \big(x''(\beta)y'(\beta)-x'(\beta)y''(\beta)\big)
\end{equation}
is the numerator of the imaginary part of $\alpha$ in \eqref{p-4}, with $a,b,x,y$ as above.
Furthermore, by \eqref{p-4}, any real zero $\beta$ of $\xi$ determines the M\"obius transformation through
\begin{equation}\label{eq:alphaofbetapol}
\alpha(\beta) = \Re\left(\displaystyle{\frac{z''_1(0)}{z'_1(0)}\cdot \frac{z'_2(\beta)}{z''_2(\beta)}}\right),
\end{equation}
as long as $\beta$ is not a zero or pole of $z'_2$ or $z''_2$. Applying this to \eqref{p-2} and \eqref{p-1} yields
\begin{equation}\label{eq:aabbofbetapol}
\bfa(\beta) = \frac{z_2'(\beta)}{z_1'(0)} \alpha(\beta), \qquad
\bfb(\beta) = z_2(\beta) - \bfa(\beta) z_1(0),
\end{equation}
determining the similarity corresponding to $\beta$.

Finally, coming back to \eqref{ec-1}, we can check whether there exists a real zero $\beta$ of $\xi$, such that
\begin{equation}\label{eq:fundrelbeta}
z_2\big(\alpha(\beta)t + \beta\big) - \bfa(\beta) z_1(t) - \bfb(\beta) = 0
\end{equation}
for all $t$. This is a rational function $R(t) = P_1(t)/P_2(t)$, whose coefficients are rational functions of $\beta$. For \eqref{eq:fundrelbeta} to hold for all $t$, the coefficients of the numerator $P_1(t)$, which are rational functions of $\beta$, have to be zero. Taking the GCD of the numerators of these coefficients and of $\xi(\beta)$ gives a polynomial $P(\beta)$. Finally let $Q(\beta)$ be the result of taking out from $P(\beta)$ all the factors that make the numerators and denominators of $\alpha(\beta), \bfa(\beta)$ and the denominator of $\bfb(\beta)$ vanish. Analogously, a polynomial $Q(\beta)$ can be found in the case of an orientation reversing similarity.

In the theorem below, let $Q(\beta)$ be derived as in this section under the assumption that the M\"obius transformation $\varphi$ is linear affine.

\begin{theorem}\label{th-Q}
The curves $\CCC_1,\CCC_2$ are similar if and only if $Q(\beta)$ has at least one real root. More than that, in that case they are symmetric if and only if $Q(\beta)$ has several real roots.
\end{theorem}

\begin{proof}
If $\CCC_1,\CCC_2$ are similar, then Theorem \ref{th-fund} implies that there exists a similarity $f$ and a M\"obius transformation $\varphi$ for which either \eqref{ec-1} or \eqref{ec-2} holds. By construction, and since $\alpha,\beta$ can be taken real, this implies that $Q(\beta)$ has a real zero. Conversely, if $\beta$ is a real zero of $Q(\beta)$, then $\alpha=\alpha(\beta)$ in \eqref{eq:alphaofbetapol} is also real, \eqref{eq:alphaofbetapol} -- \eqref{eq:aabbofbetapol} are well defined, and $f(z) = \bfa(\beta) z + \bfb(\beta)$ is a similarity that satisfies \eqref{ec-1} with $\varphi(t) = \alpha(\beta)t + \beta$. Theorem \ref{th-fund} then implies that $\CCC_1, \CCC_2$ are similar. The second statement follows from Theorem \ref{thsim}.
\end{proof}

Observe that $Q(\beta)$ cannot be identically zero: Indeed, in that case by Theorem \ref{th-Q} we would have infinitely many similarities between $\CCC_1$ and $\CCC_2$, therefore contradicting Theorem \ref{thsim}. The number of real roots of $Q(\beta)$ can be found using Sturm's Theorem \cite[\S 2.2.2]{Basu}. Additionally, there are many efficient algorithms for isolating the real roots of a polynomial, see for instance \cite[Algorithm 10.4]{Basu}, that are implemented in most popular computer algebra systems.

Thus we obtain a recipe, spelled out as Algorithm {\SimilarPol}, for detecting whether two rational curves are related by a similarity whose corresponding M\"obius transformation is linear affine, which is the case for all similar polynomial curves.

\begin{algorithm}
\begin{algorithmic}[1]
\REQUIRE Two proper rational parametrizations $z_1, z_2: \RR\dashrightarrow \CC$ of curves $\CCC_1,\CCC_2$ of equal degree greater than one, such that $z_1$ is well defined at $t=0$ and $z_1', z_1''$ are nonzero at $t=0$.
\ENSURE Returns whether $\CCC_1, \CCC_2$ are related by a similarity that corresponds to a linear affine change of the parameter downstairs in \eqref{eq:fundamentaldiagram}.
\STATE \emph{Look for orientation preserving similarities.}
\STATE Let $\xi(\beta)$ be the numerator of $\Im(\alpha)$, with $\alpha$ from \eqref{p-4}.
\STATE Find $\alpha(\beta)$ from \eqref{eq:alphaofbetapol}.
\STATE Find $\bfa(\beta)$ and $\bfb(\beta)$ from \eqref{eq:aabbofbetapol}.
\STATE Let $P(\beta)$ be the GCD of $\xi(\beta)$, and all the polynomial conditions obtained when substituting $\alpha(\beta), \bfa(\beta), \bfb(\beta)$ into \eqref{ec-1}.
\STATE Let $Q(\beta)$ be the result of taking out from $P(\beta)$
any factor shared with a denominator of $\alpha(\beta),\bfa(\beta),\bfb(\beta)$ or a numerator of $\alpha(\beta), \bfa(\beta)$.
\STATE If $Q(\beta)$ has real roots, return \texttt{TRUE}.
\STATE \emph{Look for orientation reversing similarities.}
\STATE Replace $z_1\leftarrow \bar{z}_1$ and proceed as in lines 1 -- 7.
\STATE If the above step did not return \texttt{TRUE}, then return \texttt{FALSE}.
\end{algorithmic}
\caption*{{\bf Algorithm} {\SimilarPol}}
\end{algorithm}

Although we have not made it explicit in the above algorithm, one can compute the similarities mapping $\CCC_1$ to $\CCC_2$, either symbolically, i.e., in terms of a real root $\beta$ of $Q(\beta)$, or numerically, at the additional cost of approximating the real roots of $Q(\beta)$ with any desired precision.

\subsection{The rational case I: $\delta \neq 0$} \label{subcase-1}
\noindent Next assume that $\CCC_1, \CCC_2$ are not polynomial. In that case, we can no longer assume that the M\"obius transformation $\varphi$ in \eqref{eq:Moebius} is linear affine. We start considering the case when $\varphi$ satisfies $\delta\neq 0$. After performing a linear affine change of the parameter $t$ if necessary, we may and will assume that $z_1$ is well defined at $t=0$ and that $z_1', z_1''$ are nonzero at $t = 0$. Furthermore, after dividing the coefficients of $\varphi$ by $\delta$, we can assume that $\delta = 1$ and that $\alpha,\beta,\gamma$ are real.

Evaluating \eqref{ec-1} at $t = 0$ again gives \eqref{p-1}, but differentiating \eqref{ec-1} now yields
\begin{equation} \label{dif-1}
z_2'(\varphi(t))\cdot \frac{\Delta}{(\gamma t+1)^2}=\bfa z'_1(t).
\end{equation}
Evaluating \eqref{dif-1} at $t=0$, we get
\begin{equation}\label{ig-2}
z'_2(\beta)\cdot \Delta = \bfa z_1'(0),
\end{equation}
while differentiating \eqref{dif-1} and evaluating at $t=0$ gives
\begin{equation}\label{ig-4}
\Delta \big(\Delta z''_2(\beta)-2\gamma z'_2(\beta)\big)=\bfa z_1''(0).
\end{equation}
Dividing \eqref{ig-4} by \eqref{ig-2}, using that $z_1'(0), \Delta\neq 0$, and solving for $\gamma$ gives
\begin{equation}\label{gam}
\gamma = -\frac{z''_1(0)}{2z'_1(0)} + \Delta \frac{z''_2(\beta)}{2z'_2(\beta)}.
\end{equation}

Assuming $\delta = 1$ forces $\gamma$ (and also $\alpha, \beta$) to be real. Writing
\[ -\frac{z_1''(0)}{2z_1'(0)} = A + Bi,\qquad \frac{z_2''(\beta)}{2z_2'(\beta)} = C(\beta) + D(\beta)i, \]and since $\Delta$ is also real, we have $B+\Delta\cdot D(\beta)=0$, and therefore
\begin{equation}\label{eq:delta.gamma.alpha}
\Delta(\beta) = -\frac{B}{D(\beta)},\quad \gamma(\beta) = A + \Delta(\beta) C(\beta),\quad \alpha(\beta) = \Delta(\beta) + \beta\gamma(\beta).
\end{equation}
The following lemma proves that the above expressions are well defined.

\begin{lemma} \label{nozero}
If $\Im\left(\displaystyle{\frac{z''_2(\beta)}{z'_2(\beta)}}\right) = 0$ holds for infinitely many $\beta$, then $\CCC_2$ is a line.
\end{lemma}
\begin{proof}
Writing $z_2 (\beta) = x (\beta) + i y (\beta)$, one observes
\[
\frac{z''_2}{z'_2}=\frac{x''+iy''}{x'+iy'}=\frac{x''x'+y'y''}{x'^2+y'^2}+i \frac{x'y''- x''y'}{x'^2+y'^2}.
\]
Therefore
\[
\Im\left(\frac{z''_2}{z'_2}\right) =\frac{x'y''-x''y'}{x'^2+y'^2}=\frac{x'^2}{x'^2+y'^2}\cdot \frac{x'y''-x''y'}{x'^2}=\frac{x'^2}{x'^2+y'^2}\cdot \frac{\rmd}{\rmd\beta}\left(\frac{y'}{x'}\right),
\]
which is identically zero if and only if either $x'$ is identical to zero or $y'/x'$ is identical to a constant. In either case one deduces that $\CCC_2$ is a line.
\end{proof}

Hence the similarity corresponding to $\varphi$ is found from \eqref{ig-2} and \eqref{p-1}, i.e.,
\begin{equation}\label{eq:aabbfrombetagen}
\bfa(\beta) = \frac{z_2'(\beta)}{z_1'(0)} \Delta(\beta),\qquad \bfb(\beta) = z_2(\beta) - \bfa(\beta)z_1(0).
\end{equation}
Finally, substituting $\alpha(\beta), \gamma(\beta), \bfa(\beta), \bfb(\beta)$ into \eqref{ec-1}, and taking the GCD of the numerators as before, we again reach a polynomial condition $P(\beta) = 0$. Let $Q(\beta)$ be the polynomial obtained from $P(\beta)$ by taking out the factors that make a denominator of $\Delta(\beta), \alpha(\beta), \gamma(\beta), \bfa(\beta), \bfb(\beta)$ or a numerator of $\Delta(\beta), \bfa(\beta)$ vanish. We conclude that Theorem \ref{th-Q} also holds for similarities of this type, with $Q(\beta)$ defined under the assumption that the M\"obius transformation $\varphi$ has coefficient $\delta\neq 0$.

\subsection{The rational case II: $\delta=0$} \label{subcase-2}
\noindent The remaining case happens when $\CCC_1, \CCC_2$ are related by a similarity whose corresponding M\"obius transformation \eqref{eq:Moebius} satisfies $\delta = 0$. Then $\gamma\neq 0$, and we may and will assume $\gamma = 1$ and therefore that $\alpha,\beta$ are real. Equation \eqref{ec-1} becomes
\[ z_2(\alpha + \beta/t) = \bfa z_1(t) + \bfb. \]
Making the change of parameter $t\longmapsto 1/t$, and defining $\tilde{z}_1(t) := z_1(1/t)$, we get \[z_2(\alpha+\beta t)=\bfa \tilde{z}_1(t)+\bfb\]
It follows that $\CCC_1, \CCC_2$ are related by a similarity that corresponds to a linear affine change of parameters downstairs, which is the case of Section \ref{subsec-poly-case}, with $\tilde{z}_1$ satisfying the same conditions as $z_1$. The similarities can then be detected by applying Algorithm {\SimilarPol}.

We thus arrive at a recipe for detecting whether two rational curves are similar, spelled out as Algorithm {\SimilarGen}. As in the case of {\SimilarPol}, the corresponding similarities can be computed explicitly.

\begin{algorithm}
\begin{algorithmic}[1]
\REQUIRE Two proper rational parametrizations $z_1,z_2:\RR\dashrightarrow \CC$ of curves $\CCC_1, \CCC_2$ of equal degree, none of which is a line or a circle, such that $z_1(t), z_1(1/t)$ are well defined at $t=0$ and $z_1'(t), z_1''(t), \frac{\rmd}{\rmd t} z_1(1/t), \frac{\rmd^2}{\rmd t^2} z_1(1/t)$ are nonzero at $t = 0$.
\ENSURE Returns whether $\CCC_1, \CCC_2$ are related by a similarity.
\STATE \emph{Look for orientation preserving similarities with $\delta\neq 0$.}
\STATE Find $\Delta(\beta), \gamma(\beta), \alpha(\beta)$ from \eqref{eq:delta.gamma.alpha}.
\STATE Find $\bfa(\beta), \bfb(\beta)$ from \eqref{eq:aabbfrombetagen}.
\STATE Let $P(\beta)$ be the GCD of all the polynomial conditions obtained when substituting $\alpha(\beta),\gamma(\beta),\bfa(\beta),\bfb(\beta)$ in \eqref{ec-1}.
\STATE Let $Q(\beta)$ the result of taking out from $P(\beta)$ any factor shared with a denominator of $\Delta(\beta),\alpha(\beta),\gamma(\beta),\bfa(\beta),\bfb(\beta)$ or a numerator of $\Delta(\beta), \bfa(\beta)$.
\STATE If $Q(\beta)$ has a real root, return \texttt{TRUE}.
\STATE \emph{Look for orientation reversing similarities with $\delta \neq 0$.}
\STATE Replace $z_1\leftarrow \bar{z}_1$ and proceed as in lines 1 -- 6.
\STATE \emph{Look for the remaining similarities with $\delta=0$.}
\STATE Replace $(\alpha,\beta) \leftarrow (\beta, \alpha)$ and $z_1(t) \leftarrow z_1(1/t)$.
\STATE Return the result of {\SimilarPol}.
\end{algorithmic}
\caption*{{\bf Algorithm} {\SimilarGen}}
\end{algorithm}

\subsection{The similarity type}\label{sec:similaritytype}
\noindent Notice that if two curves $\CCC_1, \CCC_2$ are identified as similar by Algorithm {\SimilarPol} or {\SimilarGen}, and if $\beta$ is a real root of $Q$, then $\bfa(\beta)$ and $\bfb(\beta)$ define a similarity transforming $\CCC_1$ into $\CCC_2$. The nature of the similarity $f$ can be found as follows by analyzing the fixed points.

In case $f$ is orientation preserving, any fixed point $z_0$ satisfies $(1 - \bfa) z_0 = \bfb$. If $\bfa = 1$ and $\bfb = 0$ then every $z_0\in \CC$ is a fixed point, and $f$ is the identity. If $\bfa = 1$ and $\bfb\neq 0$ then there are no fixed points, and $f$ is a translation by $\bfb$. If $\bfa\neq 1$, then there is a unique fixed point $z_0 = \bfb/(1 - \bfa)$. Writing $\bfa = r\e^{i\theta}$, the similarity is a counter-clockwise rotation by an angle $\theta$ around the origin, followed by a scaling by $r$ and a translation by $\bfb$.

In case $f$ is orientation reversing, any fixed point $z_0$ satisfies $z_0 - \bfa \overline{z}_0 = \bfb$. If $\bfa = 1$ and $\bfb = 0$, then the $x$-axis is invariant and $f$ is a reflection in the $x$-axis. If $\bfa = 1$ and $\bfb \neq 0$ is real, then the $x$-axis is invariant and $f$ is a glide reflection, which first reflects in the $x$-axis and then translates by $\bfb$. If $\bfa = 1$ and $\Im(\bfb)\neq 0$, then there are no fixed points and $f$ is a reflection in the $x$-axis, followed by a translation by $\bfb$. If $\bfa\neq 1$, write $\bfa = r\e^{i\theta}$. If $r = 1$ then there is a line of fixed points and $f$ is an isometry that first reflects in the $x$-axis, then rotates counter-clockwise by $\theta$, and then translates by $\bfb$. If $r\neq 1$ then there is a unique fixed point $z_0 = (\bfa\overline{\bfb} + \bfb)/(1 - |\bfa|^2)$, and $f$ first reflects in the $x$-axis, rotates counter-clockwise by $\theta$ around the origin, scales by $r$, and finally translates by $b$.

\subsection{Piecewise rational curves}
\noindent Let us now assume that $\CCC_1, \CCC_2$ are piecewise rational curves. This situation is common in Computer Aided Geometric Design, for instance, where objects are usually modeled by means of (rational) B\'ezier curves or NURBS, which are piecewise rational. In this setting, $\CCC_1, \CCC_2$ are parametrized by $\bfp_1(t),\bfp_2(t)$, where
\[
\begin{array}{ccc}
\bfp_1(t) := \left\{\begin{array}{cl}
\bfx_1(t) & \mbox{if } t\in[t_1,t_2],\\
\bfx_2(t) & \mbox{if } t\in[t_2,t_3],\\
\vdots & \qquad \vdots\\
\bfx_m(t) & \mbox{if } t\in[t_m,t_{m+1}],
\end{array}\right. &  &
\bfp_2(t) := \left\{\begin{array}{cl}
\bfy_1(t) & \mbox{if } t\in[t_1,t_2],\\
\bfy_2(t) & \mbox{if } t\in[t_2,t_3],\\
\vdots & \qquad \vdots \\
\bfy_n(t) & \mbox{if } t\in[t_n,t_{n+1}],
\end{array}\right.
\end{array}\]
for some integers $m,n\geq 1$ and rational functions $\bfx_1,\ldots,\bfx_m$ and $\bfy_1,\ldots\bfy_n$.

We want to detect whether $\CCC_1, \CCC_2$ are similar, or perhaps whether $\CCC_1, \CCC_2$ are \emph{partially} similar, meaning that some part of $\CCC_1$ is similar to some part of $\CCC_2$. For this purpose we need to compare pieces of the curves, which are each described by a (rational) parametrization on a certain parameter interval, both of which need to be taken into account. Let $\bfx_i(I)$, $\bfy_j(J)$ be two such pieces on the intervals $I:=[t_i,t_{i+1}]$, $J:=[t_j,t_{j+1}]$. Then $\bfx_i(I)$ and $\bfy_j(J)$ are similar if and only if (1) the whole curves $\bfx_i(\RR)$ and $\bfy_j(\RR)$ are similar, and (2) the corresponding M\"obius transformation $\varphi$ satisfies that $\varphi(I)=J$. For $\CCC_1$ and $\CCC_2$ to be similar in this setting, every piece of $\CCC_1$ must be a similar with a piece of $\CCC_2$ and conversely, all with he same similarity. If (1) holds and (2) is replaced by the condition $J' := \varphi(I)\cap J\neq \emptyset$, then we have a partial similarity, in the sense that $\bfy_j\big(J'\big)$ is similar to $\bfx_i\big(\varphi^{-1}(J')\big)$.

\subsection{Computing symmetries}\label{sec:symmetry}
\noindent From Proposition \ref{prop:sym}, we can find the symmetries of $\CCC$ by applying Algorithms {\SimilarPol} and {\SimilarGen} with both arguments the parametrization $z_1 = z_2$ of $\CCC$. Furthermore, the nature of the symmetry can be deduced from the set of fixed points, as observed in Section \ref{sec:similaritytype}. Since an irreducible algebraic curve different from a line cannot have a nontrivial translation symmetry or glide-reflection symmetry, we are left with the cases mentioned in Section \ref{sec-prelim}.

\section{Implementation and experimentation} \label{sec-implem}
\noindent Algorithms {\SimilarPol} and {\SimilarGen} were implemented in {\Sage} \cite{sage}, using {\Singular} \cite{singular} and {\Flint} \cite{flint} as a back-end. The corresponding worksheet ``Detecting Similarity of Rational Plane Curves'' is available online for viewing at the third author's website \cite{WebsiteGeorg} and at \textsf{https://cloud.sagemath.org}, where it can be tried online once \textsf{SageMath Cloud} supports public worksheets.

\subsection{An example: computing the symmetries of the deltoid}\label{sec:deltoidsym}

\begin{figure}
\centering
\begin{subfigure}[b]{0.45\textwidth}
  \centering
  \includegraphics[scale=0.8]{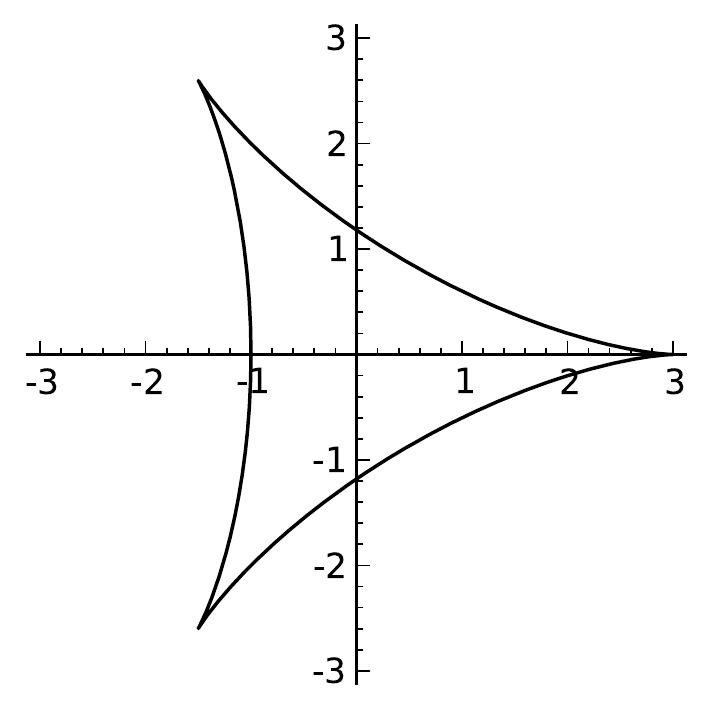}
  \caption{ }\label{fig:deltoid1}
\end{subfigure}%
\qquad
\begin{subfigure}[b]{0.45\textwidth}
  \centering
  \includegraphics[scale=0.8]{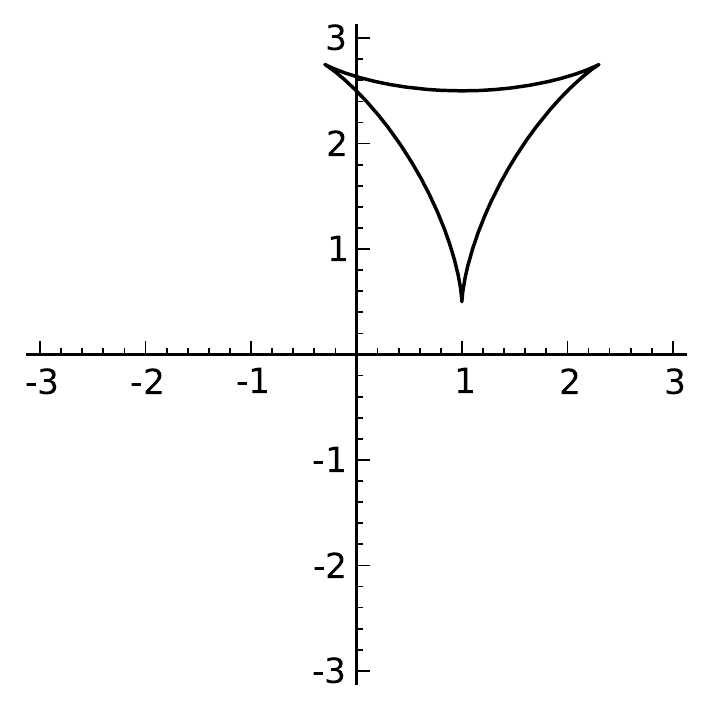}
  \caption{ }\label{fig:deltoid2}
\end{subfigure}%
\caption{The deltoid given as the image of \eqref{eq:deltoid} (left) and a related deltoid obtained by a similarity (right).}\label{fig:deltoid}
\end{figure}

\noindent Let $\CCC_1 \subset \RR^2$ be the \emph{deltoid}, defined parametrically as the image of the rational map
\begin{equation}\label{eq:deltoid}
\phi = (\phi_1, \phi_2): \RR\longrightarrow \RR^2,\qquad t\longmapsto \left(\frac{-t^4 - 6t^2 + 3}{(t^2 + 1)^2}, \frac{8t^3}{(t^2+1)^2} \right),
\end{equation}
or implicitly as the zero locus of the equation
\[ (x^2+y^2)^2 - 8(x^3-3xy^2) + 18(x^2 + y^2) - 27 = 0.\]
As can be seen in Figure \ref{fig:deltoid1}, this curve is invariant under the symmetric group $S_3$ with six elements, which is generated by a 120$^\circ$-rotation around the origin and the reflection in the horizontal axis.

Since $\phi'(0) = (0, 0)$, this curve needs to be reparametrized before we can apply the algorithm. One checks that $z_1 = \phi_1(t-1) + \phi_2(t-1)i$ satisfies the conditions of the algorithm. To detect the symmetries of $\CCC_1$, we let $z_2 := z_1$ and apply the algorithm. We find $\Delta = \frac{1}{2}(\beta^2 - 2\beta + 2)$,
\[\gamma(\beta) = -\frac{\beta(2\beta^3 - 3\beta^2 - 18\beta + 22)}{4(\beta - 1)(\beta^2 - 2\beta - 2)},\]
\[\alpha(\beta) = -\frac{7\beta^4 - 34\beta^3 + 30\beta^2 + 8\beta - 8}{4(\beta - 1)(\beta^2 - 2\beta -2)},\]
\begin{equation}\label{eq:aaofbeta}
\bfa(\beta) = \frac{(\beta - 1)(\beta^2 - 2\beta - 2)}{(\beta^2 - 2\beta + 2)^2}\big(2 - \beta - \beta i\big),
\end{equation}
\begin{equation}\label{eq:bbofbeta}
\quad \bfb(\beta) = 3\frac{\beta(\beta^2 - 6\beta + 6)}{(\beta^2 - 2\beta + 2)^2}\big(1 - (\beta - 1)i \big)
\end{equation}
as functions of $\beta$. Substituting these expressions into
\[ z_2\left(\frac{\alpha(\beta) t + \beta}{\gamma(\beta) t + 1}\right) - \bfa(\beta) z_1(t) + \bfb(\beta)\]
yields a rational function in $t$ whose coefficients are rational functions in $\beta$. Its numerator is a polynomial in $t$, and taking the GCD of the numerators of its coefficients, we find
\[\beta(\beta - 1)(\beta^2 - 6\beta + 6)(\beta^2 - 2\beta - 2).\]
We are only interested in the $\beta$ for which the expressions for $\alpha, \gamma, \Delta, \bfa$ and $\bfb$ are well-defined and for which the M\"obius transformation and the similarity are invertible, so we discard the factors corresponding to the poles of $\alpha, \gamma, \Delta, \bfa, \bfb$ and zeros of $\Delta, \bfa$, leaving
\[ \beta(\beta^2 - 6\beta + 6). \]
Substituting the three zeros of this equation into Equations \eqref{eq:aaofbeta} and \eqref{eq:bbofbeta}, we find three orientation preserving symmetries $f: z\longmapsto \bfa z + \bfb$, with $\bfa = 1, \e^{2\pi i/3}, \e^{4\pi i/3}$ and $\bfb = 0$, corresponding to rotations by 0$^\circ$, 120$^\circ$, and 240$^\circ$. Repeating the procedure with $z_1\mapsto \overline{z}_1$, we find the remaining three orientation reversing symmetries $f: z\longmapsto \bfa \overline{z} + \bfb$ in $S_3$ with again $\bfb = 0$ and $\bfa = 1, \e^{2\pi i/3}, \e^{4\pi i/3}$ and $\bfb = 0$.

\subsection{An example: detecting similarities between two deltoids}
\noindent Next, we let $z_1$ be as in Section \ref{sec:deltoidsym} and define
\[ z_2: \RR\longrightarrow \CC,\qquad t\longmapsto \frac{t^4 + 4t^3 + 2t^2 + 1}{(t^2 + 1)^2} + \frac{5t^4 + 14t^2 + 1}{2(t^2 + 1)^2} i.
\]
The corresponding curves $\CCC_1$ and $\CCC_2$ are shown in Figure \ref{fig:deltoid}. Applying the Algorithm {\SimilarGen}, we obtain three orientation preserving similarities $f(z) = \bfa z + \bfb$ and three orientation reserving similarities $f(z)=\bfa \overline{z} + \bfb$, each with
$\bfb = 1 + 2i$ and $\bfa \in \big\{-\frac{1}{2} i, \frac{1}{2}\e^{\pi i/6}, \frac{1}{2}\e^{5\pi i/6}\big\}$.
A direct computation yields
\[ z_2(t - 1) = -\frac{1}{2} i\big(\phi_1(t - 1) + \phi_2(t - 1) i\big) + 1 + 2i = -\frac{1}{2}iz_1(t) + 1 + 2i,\]
confirming that \eqref{ec-1} holds with $\bfa = -i/2, \bfb = 1 + 2i$ and $\varphi(t) = t - 1$.

\subsection{Observations on the complexity}
\noindent For simplicity let us determine the complexity of Algorithm {\SimilarPol}; the analysis of Algorithm {\SimilarGen} is completely analogous. In addition to the standard Big O notation $\OOO$, we use the \emph{Soft O} notation $\tOOO$ to ignore any logarithmic factors in the complexity analysis.

Let $z_1,z_2$ be polynomials of degree at most $d$. Lines 1--4 can be carried out in $\OOO(d)$ (integer) operations, resulting in a polynomial $\xi$ and rational functions $\alpha, \bfa$ and $\bfb$ whose numerators and denominators have degrees $\OOO(d)$. Observe that the bottleneck is line 5, where we substitute $\alpha(\beta)$ into Equation \eqref{ec-1}. Write
\[
\alpha(\beta)
 = \Re\left( \frac{z''_1(0)}{z'_1(0)}\cdot \frac{z'_2(\beta)}{z''_2(\beta)}\right)
 =: \frac{\alpha_1(\beta)}{\alpha_2(\beta)},
\]
for some relatively prime polynomials $\alpha_1, \alpha_2$ of degree at most $2d - 3$. For \eqref{eq:fundrelbeta} to hold identically for all $t$, we write $z_2(t) = c_d t^d+\cdots +c_1t+c_0$ and compute
\begin{align*}
\sum_{l=0}^d P_l(\beta) t^l
  & := \alpha_2^d(\beta) \Big[ z_2\big(\alpha(\beta)t+\beta\big) - \bfa(\beta) z_1(t) - \bfb(\beta) \Big]\\
  & = \sum_{l=0}^d \left[\sum_{k = l}^d c_k {l\choose k}\alpha_2^{d-l}(\beta) \alpha_1^l(\beta) \beta^{k-l}\right] t^l - \alpha_2^d(\beta) \Big[\bfa(\beta) z_1(t) + \bfb(\beta) \Big].
\end{align*}

Each polynomial $P_l$ has degree $\OOO(d^2)$, and can be computed in $\tOOO(d^3)$ operations by using binary exponentiation and FFT-based multiplication \cite[\S 8.2]{VonZurGathen.Gerhard}. The polynomial $\xi$ from \eqref{eq:xi} has degree at most $2d - 3$ and can be found in $\tOOO(d)$ operations. Thus the polynomials $\xi, P_0, P_1,\ldots, P_d$ can be computed in $\tOOO(d^4)$ operations. The GCD of two polynomials of degree at most $d$ can be computed in $\tOOO(d)$ operations \cite[Corollary 11.6]{VonZurGathen.Gerhard}. The GCD of the polynomials $\xi, P_0, P_1,\ldots, P_d$ can therefore be computed in $\tOOO(d^3)$ operations and has degree $\OOO(d)$. At line 6 one takes the GCD $Q$ of polynomials of degree $\OOO(d)$, which requires $\tOOO(d)$ operations. To check, at the next line, whether $Q$ has real roots can be done by root isolation, which takes $\tOOO(d^3)$ operations \cite[Theorem 17]{KS}. We conclude that the overall complexity is $\tOOO(d^4)$. The analysis of Algorithm {\SimilarGen} is cumbersome, but can be carried out similarly.

Of course the theoretical complexity is only a proxy for the practical performance of the algorithm, which also depends on the architecture, the precise implementation, and the bitsizes of the coefficients of the parametrizations.

\subsection{Performance}
\noindent Let the \emph{bitsize} $\tau := \lceil \log_2 k \rceil + 1$ of an integer $k$ be the number of bits needed to represent it. For various degrees $d$ and bitsizes $\tau$, we test the performance of Algorithms {\SimilarPol} and {\SimilarGen} by applying them to a pair of parametrized curves of equal degree $d$ and bitsizes of their coefficients bounded by $\tau$. For every pair $(d,\tau)$ this is done a number of times, and the average execution time $t$ (CPU time) is listed in Tables \ref{tab:performance_pol} and \ref{tab:performance_gen} for a Dell XPS 15 laptop, with 2.4 GHz i5-2430M processor and 6 GB RAM.

Double logarithmic plots of the CPU times against the degrees and against the coefficient bitsizes are presented in Figures \ref{fig:CPU-pol} and \ref{fig:CPU-gen}. In these plots the CPU times lie on a curve that seems to asymptotically approach a straight line, suggesting an underlying power law. The least-squares estimates of these asymptotes gives the power laws
\begin{equation}\label{eq:leastsquares}
t_\mathrm{pol} = 2.17\cdot 10^{-11} d^{9.28} \tau^{1.59},\qquad
t_\mathrm{gen} = 6.2\cdot 10^{-06} d^{9.91} \tau^{1.91}
\end{equation}
for the CPU times $t_\mathrm{pol}$ of {\SimilarPol} and $t_\mathrm{gen}$ of {\SimilarGen}.

We should emphasize that these timings are for \emph{dense} polynomials. The performance is much better in practice than Tables \ref{tab:performance_pol}, \ref{tab:performance_gen} and Figures \ref{fig:CPU-pol}, \ref{fig:CPU-gen} suggest. To illustrate this, we find the symmetries of various well-known curves $z_1: \RR\longrightarrow \CC$, namely the \emph{folium of Descartes},
\[ t\longmapsto 3t \frac{1 + t i}{1+t^3}, \]
the \emph{lemniscate of Bernoulli},
\[ t\longmapsto \frac{(3t^4 + 2t^3 - 2t - 3) + (t^4 + 6t^3 - 6t - 1)i}{5t^4 + 12t^3 + 30t^2 + 12t + 5},\]
an \emph{epitrochoid},
\[ t\longmapsto \frac{(-7t^4+288t^2 + 256) + (-80t^3 + 256t)i}{t^4 + 32t^2 + 256},\]
an \emph{offset curve to a cardioid},
\begin{align*}
t\longmapsto
&  \frac{6t^8 - 756t^6 + 3456t^5 - 31104t^3 + 61236t^2 - 39366}{t^8 + 36t^6 + 486t^4 + 2916t^2 + 6561} \\
& \qquad - \frac{18t(6t^6-16t^5-126t^4+864t^3-1134t^2-1296t+4374)}{t^8 + 36t^6 + 486t^4 + 2916t^2 + 6561}i,
\end{align*}
a \emph{hypocycloid} of degree 8,
\begin{align*} t\longmapsto
& \frac{-3 t^{8} - 24 t^{7} - 120 t^{6} - 384 t^{5} - 680 t^{4} - 608 t^{3} - 224 t^{2} + 16}{t^{8} + 8 t^{7} + 32 t^{6} + 80 t^{5} + 136 t^{4} + 160 t^{3} + 128 t^{2} + 64 t + 16}\\
& \qquad + \frac{16 t^{7} + 112 t^{6} + 304 t^{5} + 400 t^{4} + 320 t^{3} + 256 t^{2} + 192 t + 64}{t^{8} + 8 t^{7} + 32 t^{6} + 80 t^{5} + 136 t^{4} + 160 t^{3} + 128 t^{2} + 64 t + 16}i
,\end{align*}
and \emph{Rose curves}
\[
t\longmapsto \frac{2t + (1-t^2)i}{(1+t^2)^{n+1}} \sum_{k=0}^n \binom{2n}{2k} (-1)^k t^{2k},\qquad n = 2, 4, 6, 8, 10,
\]
of degrees 6, 10, 14, 18, and 22. In each case the average CPU time used by {\SimilarGen} is listed in Table \ref{tab:example-curves}. Observe that the degrees and coefficients of these examples are far from trivial. For more details we refer to the {\Sage} worksheet \cite{WebsiteGeorg}.

To conclude, notice that the order of $z_1$ and $z_2$ is important for the performance of {\SimilarGen}, so that one should let $z_2$ be the parametrization whose coefficients have the smallest bitsize.

\begin{table}
\begin{center}
\begin{tabular*}{\columnwidth}{l@{\extracolsep{\stretch{1}}}*{6}{r}@{}}
\toprule
$t_{\mathrm{pol}}\qquad$ & $\tau=1$ & $\tau=8$ & $\tau=16$ & $\tau=32$ & $\tau=64$ & $\tau=128$\\
\midrule
$d =  3$ &  0.024 &  0.022 &    0.024 &    0.020 &    0.030 &     0.038\\
$d =  6$ &  0.076 &  0.198 &    0.248 &    0.336 &    0.650 &     1.464\\
$d =  9$ &  0.764 &  2.316 &    3.220 &    6.626 &   15.151 &    42.753\\
$d = 12$ &  3.188 & 14.589 &   24.902 &   61.318 &  159.452 &   448.476\\
$d = 15$ & 11.885 & 90.340 &  206.469 &  587.249 & 1244.176 &  4439.909\\
\bottomrule
\end{tabular*}
\end{center}
\caption{Average CPU time $t_\text{pol}$ (seconds) of {\SimilarPol} applied to random polynomial parametrizations of given degree $d$ and with integer coefficients with bitsizes bounded by $\tau$.}\label{tab:performance_pol}
\end{table}

\begin{table}
\begin{center}
\begin{tabular*}{\columnwidth}{l@{\extracolsep{\stretch{1}}}*{6}{r}@{}}
\toprule
$ t_{\text{gen}}\qquad$  &  $\tau=1$ &  $\tau=2$ &  $\tau=4$ &  $\tau=8$ & $\tau=16$ & $\tau=32$ \\ 
\midrule
$d=2$  &      0.33 &      0.43 &      0.64 &      1.40 &      3.26 &      9.76 \\ 
$d=3$  &      1.50 &      3.90 &      8.53 &     21.94 &     85.00 &    387.69 \\ 
$d=4$  &      8.93 &     19.60 &     61.93 &    325.01 &    724.95 &   3881.44 \\ 
$d=5$  &     41.66 &     96.12 &    380.02 &   3956.86 &  22608.96 &           \\ 
$d=6$  &     99.05 &    452.70 &   1028.22 &  17606.67 &           &           \\ 
$d=7$  &    250.39 &   2670.24 &  19478.14 &           &           &           \\ 
$d=8$  &   1805.01 &  21610.26 &           &           &           &           \\ 
\bottomrule
\end{tabular*}
\end{center}
\caption{Average CPU time $t_\text{gen}$ (seconds) of {\SimilarGen} applied to random rational parametrizations of given degree $d$ and with integer coefficients with bitsizes bounded by $\tau$.}\label{tab:performance_gen}
\end{table}

\begin{table}
\begin{center}
\begin{tabular*}{\columnwidth}{l@{\extracolsep{\stretch{1}}}*{6}{c}@{}}
\toprule
curve & Descartes'& Bernoulli's & epitrochoid & cardioid & hypocycloid\\
& folium & \!lemniscate\!& & offset & \\
 &
\includegraphics[scale=0.7]{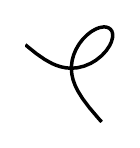}  & \includegraphics[scale=0.7]{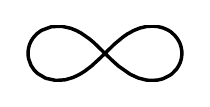} &  \includegraphics[scale=0.7]{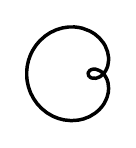} & \includegraphics[scale=0.7]{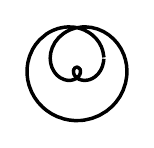} & \includegraphics[scale=0.7]{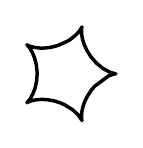} \\
degree   &          3 &                 4 &           4 &               8 &    8 \\
CPU time &       0.11 &              0.32 &        0.11 &            0.77 & 3.59 \\
\midrule
curve & 4-leaf rose & 8-leaf rose & 12-leaf rose & 16-leaf rose & 20-leaf rose \\
 &
\includegraphics[scale=0.7]{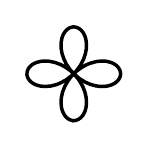} & \includegraphics[scale=0.7]{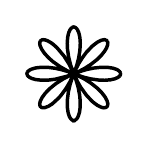} &  \includegraphics[scale=0.7]{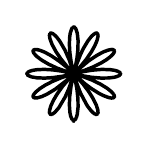}  & \includegraphics[scale=0.7]{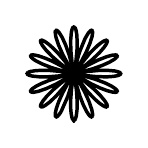} & \includegraphics[scale=0.7]{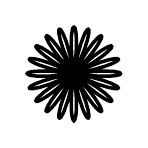} \\
degree      & 6 & 10 & 14 & 18 & 22 \\
CPU time    & 0.24 & 3.50 & 24.83 & 118.74 & 703.12 \\
\bottomrule
\end{tabular*}
\end{center}
\caption{Average CPU time (seconds) of {\SimilarGen} for well-known curves.}\label{tab:example-curves}
\end{table}

\begin{figure}
\begin{subfigure}[b]{0.45\textwidth}
  \includegraphics[scale=0.6]{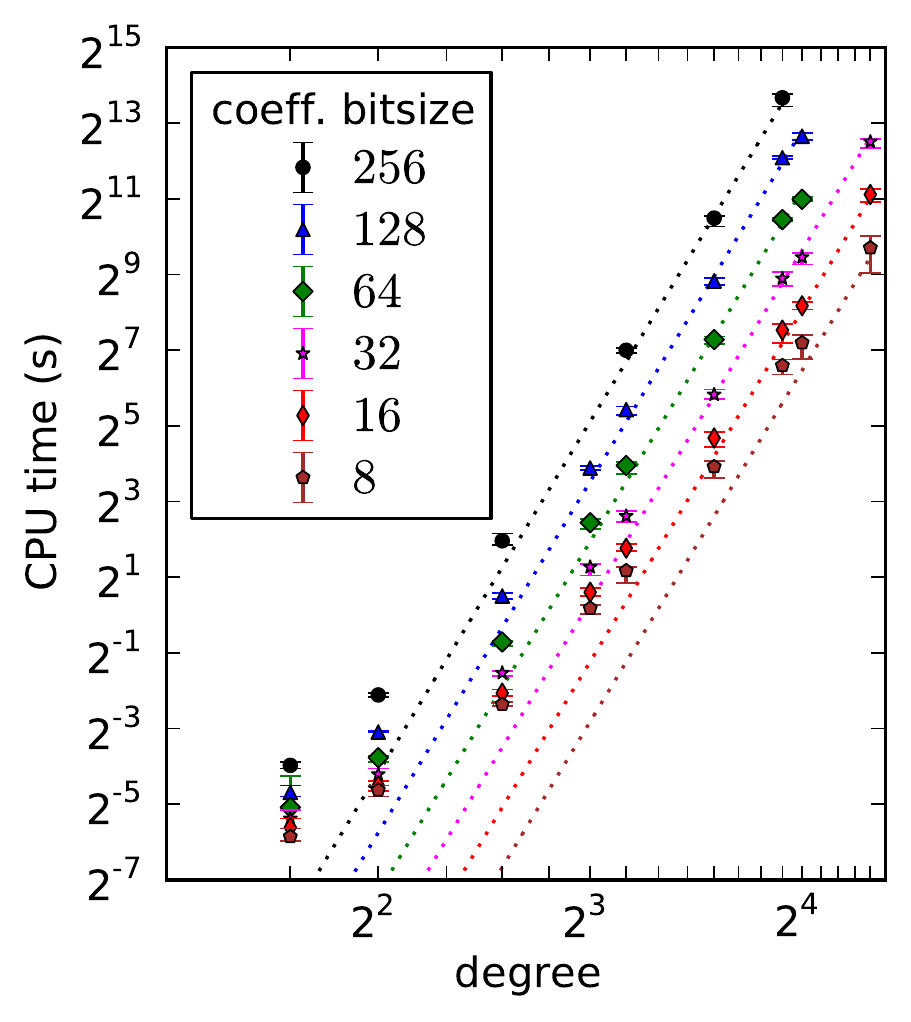}
  \caption{ }\label{fig:CPU-vs-degree}
\end{subfigure}%
\qquad
\begin{subfigure}[b]{0.45\textwidth}
  \centering
  \includegraphics[scale=0.6]{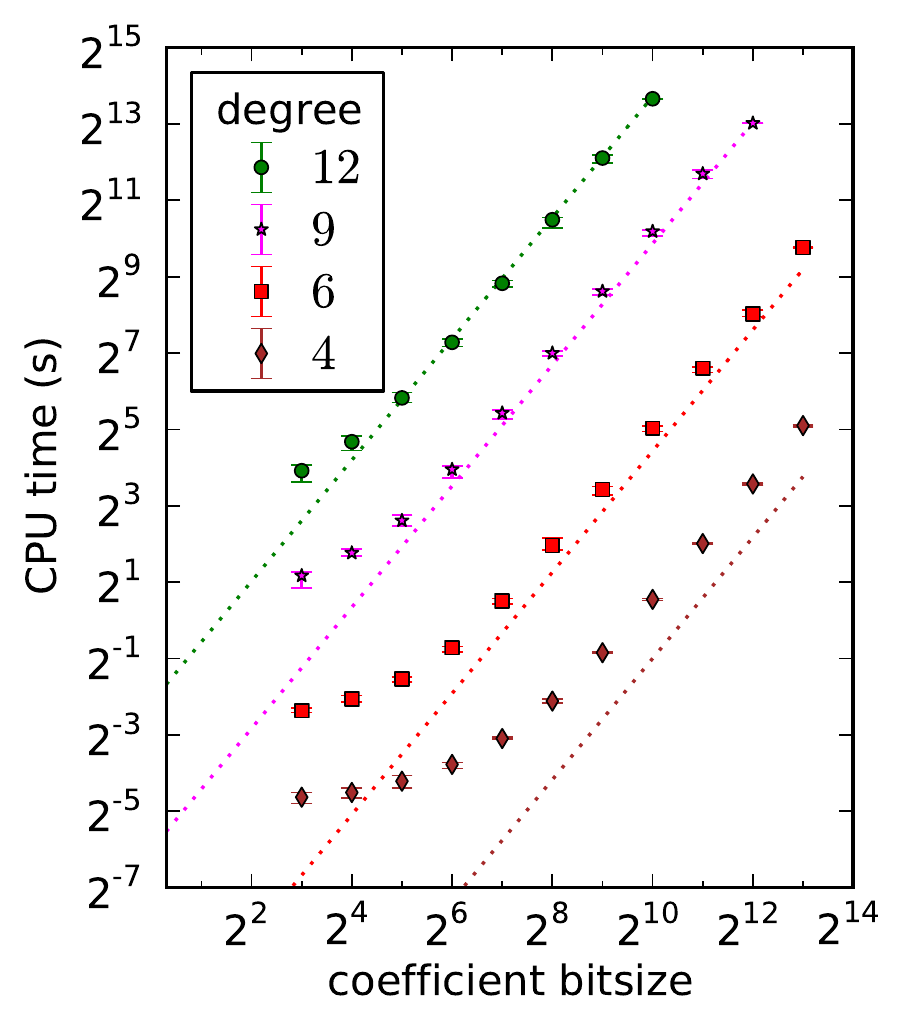}
  \caption{ }\label{fig:CPU-vs-coeff}
\end{subfigure}%
\caption{Double logarithmic plots of the average CPU time of {\SimilarPol} versus degree (left) and versus the bitsize of the coefficients (right). The error bars show the range of CPU times found for the various random polynomials. The dotted line represents the fitted power law for $t_\mathrm{pol}$ in \eqref{eq:leastsquares}.}\label{fig:CPU-pol}
\end{figure}

\begin{figure}
\begin{subfigure}[b]{0.45\textwidth}
  \includegraphics[scale=0.6]{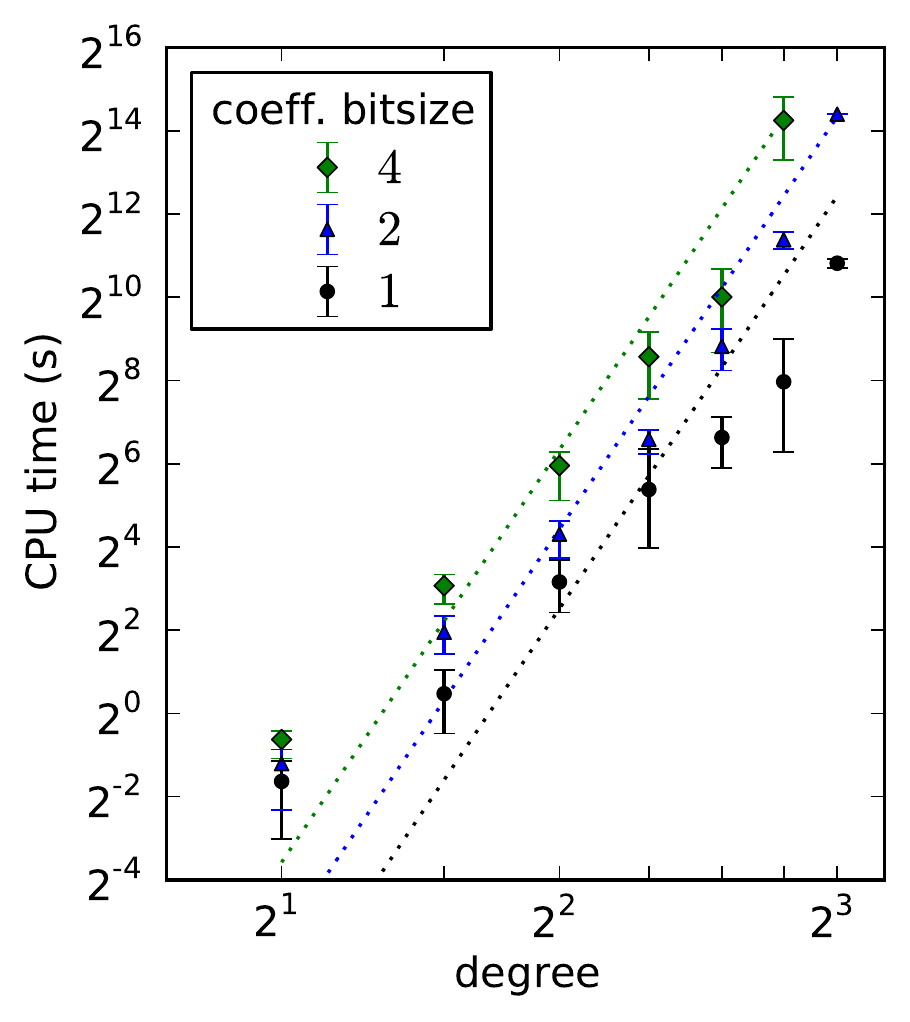}
  \caption{ }\label{fig:CPU-vs-degree-gen}
\end{subfigure}%
\qquad
\begin{subfigure}[b]{0.45\textwidth}
  \centering
  \includegraphics[scale=0.6]{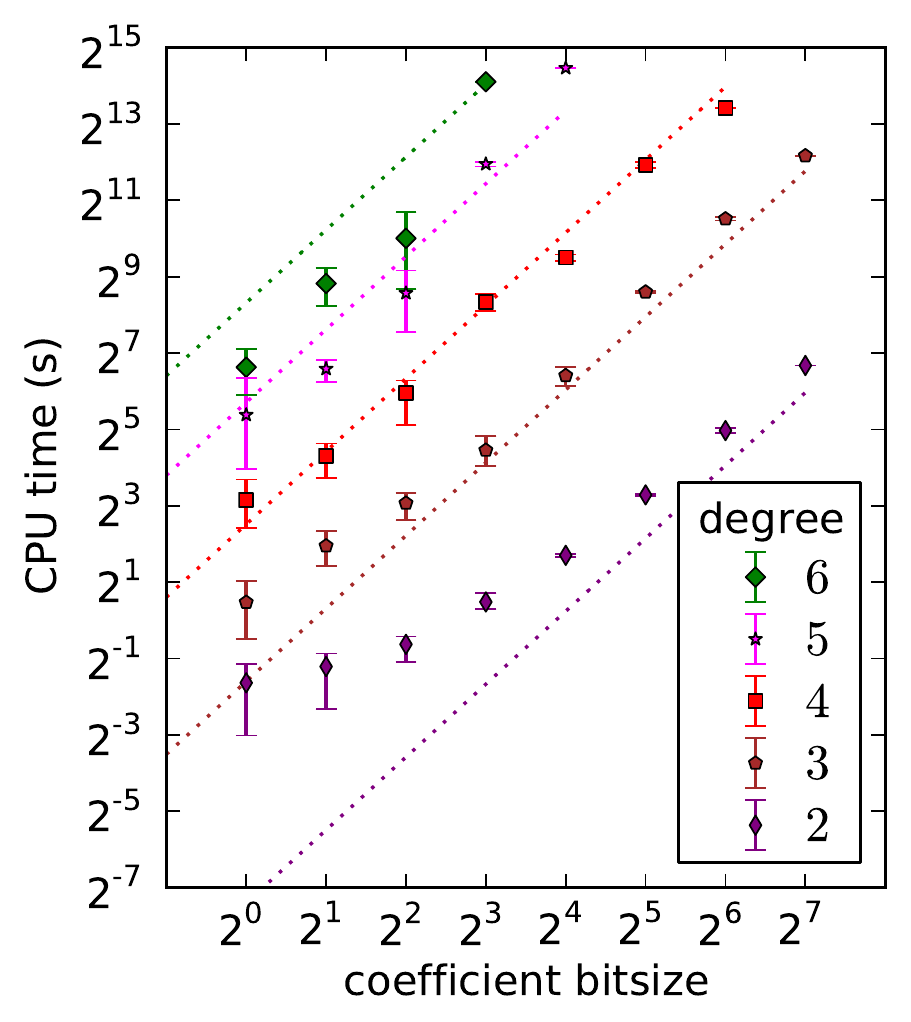}
  \caption{ }\label{fig:CPU-vs-coeff-gen}
\end{subfigure}%
\caption{Double logarithmic plots of the average CPU time of {\SimilarGen} versus degree (left) and versus the bitsize of the coefficients (right). The error bars show the range of CPU times found for the various random polynomials. The dotted line represents the fitted power law for $t_\mathrm{gen}$ in \eqref{eq:leastsquares}.}\label{fig:CPU-gen}
\end{figure}

\newpage

\end{document}